\documentclass{article}
\usepackage[margin=1in]{geometry}
\usepackage[utf8]{inputenc}

\usepackage{amsmath}
\usepackage{bm}
\usepackage{amssymb}
\usepackage{graphicx}
\usepackage{xcolor}

\usepackage[
backend=biber,
style=numeric,
sorting=ynt    
]{biblatex}
\usepackage{booktabs}

\usepackage{mathtools}

\newcounter{myalg}
\usepackage{amsthm}
\theoremstyle{plain}
\newtheorem{theorem}{Theorem}
\newtheorem{corollary}{Corollary}
\newtheorem{proposition}{Proposition}

\theoremstyle{definition}
\newtheorem{definition}{Definition}
\theoremstyle{remark}
\newtheorem{remark}{Remark}

\DeclareMathOperator\supp{supp}

\newcommand{\w}{{\bm w}}
\newcommand{\x}{{\bm x}}
\newcommand{\X}{{\bm X}}
\newcommand{\z}{{\bm z}}
\newcommand{\Z}{{\bm Z}}
\newcommand{\s}{{\bm s}}
\newcommand{\uu}{{\bm u}}
\newcommand{\vv}{{\bm v}}

\newcommand{\y}{{\bm y}}
\newcommand{\Y}{{\bm Y}}
\newcommand{\D}{{\bm D}}

\title{On the Order-Conditional Optimality of Gaffke's Bound}
\author{
  George Bissias \qquad Erik Learned-Miller \\
  University of Massachusetts Amherst
}
\date{July 24, 2026}

\begin{document}

\maketitle

\begin{abstract}
    Let $\X = (X_1, \ldots, X_n)$ be a random vector from any Borel probability law on $\mathbb{R}_+^n$. We revisit the problem of deriving a lower confidence bound (LCB) on a scalar parameter of that law.  We recast classical work, beginning with Buehler~\cite{Buehler1957}, in purely probabilistic terms to form a more accessible and extensible framework. We then specialize the framework to the case where the components of $\X$ are independent. In this context, we prove that Gaffke's bound is Buehler optimal for the order that it induces with respect to the maximum marginal mean parameter: $\max_{i \in [n]} E_Q[X_i]$, which reduces to the common mean when the $X_i$ are independent and identically distributed. That is to say,  no other valid LCB that orders samples in the same way as Gaffke's bound can improve on it with respect to this parameter. 
\end{abstract}

\section{Introduction}

We develop a general framework for constructing and comparing one-sided confidence bounds for a scalar parameter of a joint probability law from an observed vector $\X = (X_1, \ldots, X_n) \in \mathbb{R}_+^n$. Here, a parameter can be any prespecified scalar function of the probability law. The framework does not require the components of $\X$ to be independent or identically distributed. Given a class of conceivable joint laws, a scalar parameter of those laws, and a total preorder on the sample space, the framework identifies a conditionally optimal confidence bound among all bounds consistent with that preorder. This construction generalizes the development of order-determined bounds of Learned-Miller~\cite{learnedmiller2025admissibilityboundsmeandiscrete}. We then apply this framework to establish the conditional optimality of Gaffke’s bound~\cite{gaffke2005three} on the mean of the individual components of $\X$ in the case when they are independent and identically distributed. We further generalize this result in two ways. First, we show that if the components of $\X$ share the same mean and are independent, but are not identically distributed, then Gaffke's bound remains conditionally optimal when the parameter of interest is their common mean. Second, when the components of $\X$ are merely independent, but do not necessarily share the same mean, we show that Gaffke's bound is conditionally optimal when the parameter of interest is the maximum mean among all those components. 

Our results complement recent findings of Ming et al.~\cite{ming2026gaffke}, who show that, for independent bounded observations having a common mean, Gaffke’s confidence interval is inadmissible: there exists another valid confidence-interval procedure whose interval is contained in Gaffke’s interval at every input and is strictly smaller at some inputs. This does not conflict with our result, which establishes optimality only among lower confidence bounds consistent with the total preorder induced by Gaffke’s bound.

The construction developed below belongs to the classical lineage of Neyman’s inversion principle~\cite{Neyman1937} and Buehler’s theory of order-constrained confidence bounds~\cite{Buehler1957,GuerreroDavid1985}. We nevertheless develop it from first principles, beginning only with the requirement that a data-dependent lower bound exceed the true parameter with probability at most $\alpha$. From this elementary guarantee, we introduce model-indexed exceptional regions, show how they generate valid bounds by optimization over the models not excluded by the observation, and then determine the largest valid bound that respects a prescribed ordering of the sample space. In contrast to classical literature, our approach is purely probabilistic and begins with a total preorder over the sample space instead of an ordering statistic. As a result, it allows us to capture bounds arising from orders that cannot be expressed in terms of a threshold statistic. For example, when $n \geq 2$, the lexicographic order on $\mathbb{R}^n_+$ is not induced by any scalar statistic $T: \mathbb{R}^n_+ \rightarrow \bar{\mathbb{R}}$~\cite{mandler2021lexicographic}. We note specifically that there \emph{does} exist such a representation when the sample space is countable.

Other recent works have sought similar ends using the work of Learned-Miller~\cite{learnedmiller2025admissibilityboundsmeandiscrete} as a point of departure.  Bissias~\cite{bissias2026algorithmsapproximatingconditionallyoptimal} extended the notion of conditional optimality to total preorders, but only in a finite sample setting. Phan and Learned-Miller~\cite{phan2026automatedconfidenceboundprovers} derived similar results to the present work in the more general context of continuous and unbounded sample spaces. However, their bounds are limited to i.i.d.~laws and use the classical approach of ordering according to a sample statistic. 

The main application builds on Vlassis and Thomas~\cite{vlassis2026exact}, who proved finite-sample validity of Gaffke’s statistic for independent, not necessarily identically distributed, nonnegative random variables. Here we show that, for every fixed $\alpha \in (0,1)$, Gaffke’s lower confidence bound for the maximum marginal mean coincides with the nested Buehler bound generated by its own sample preorder. We also show that the defining infimum can be approached, at each sample, by an explicit family of i.i.d. distributions supported on zero and one additional value.

The same equality and conditional-optimality conclusion extends to any product-law subclass having aggregate support $\mathbb{R}_+^n$ and containing every i.i.d. law whose common marginal is supported on $\{0, s\}$, for each $s \geq 0$. In particular, it applies to the i.i.d. model and to the independent common-mean model. In the latter model, the maximum marginal mean is simply the common marginal mean.

\section{Preliminaries}

For any $n \geq 1$, let $\mathcal{Q}$ be any collection of probability measures on the measurable space $(\mathbb{R}_+^n, \mathcal{B}^n)$, where $\mathbb{R}_+$ denotes the half-line of non-negative reals and $\mathcal{B}^n$ is the Borel $\sigma$-algebra on $\mathbb{R}_+^n$. By $\bar{\mathbb{R}}$ we denote the extended reals, i.e. $\bar{\mathbb{R}} \coloneqq \{-\infty\} \cup \mathbb{R} \cup \{\infty\}$. For each $Q \in \mathcal{Q}$, let $\pi[Q] \in \bar{\mathbb{R}}$ denote a specified scalar parameter of joint law $Q$ and let $\X = (X_1, \ldots, X_n) \sim Q$ denote a random sample from $Q$. Parameter $\pi[Q]$ need not be finite. 
For convenience, we define the \emph{closed model support}, $\Omega \subseteq \mathbb{R}_+^n$, by
\begin{equation*}
    \Omega \coloneqq \overline{\bigcup_{Q \in \mathcal{Q}} \supp(Q)},
\end{equation*}
which is the smallest closed subset of $\mathbb{R}^n$ such that $\forall Q \in \mathcal{Q}$, $Q[\Omega] = 1$. Unless an additional restriction is stated explicitly, the components of $\X$ are not assumed to be independent or identically distributed. Finally, we adopt standard notation for sets of natural numbers: $[i] \coloneqq \{1, \ldots, i\}$. When $i$ is less than 1, $[i]$ is defined to be empty.

\begin{definition}
\label{def:total_preorder}
    A total preorder $D$, defining the binary operator $\lesssim_D$, is characterized by the following properties
    \begin{itemize}
        \item $\forall \x \in \Omega, \x \lesssim_D \x$.
        \item $\forall \x, \y, \z \in \Omega, \x \lesssim_D \y$ and $\y \lesssim_D \bm z$ implies that $\x \lesssim_D \bm z$.
        \item $\forall \x, \y \in \Omega, \x \lesssim_D \y$ or $\y \lesssim_D \x$.
    \end{itemize}
    When $\x \lesssim_D \y$ and $\y \lesssim_D \x$ we write $\x \sim_D \y$. And by $\x <_D \y$ we indicate that $\x \lesssim_D \y$ and $\x \not \sim_D \y$. Total preorders are total orders over the equivalence classes defined by $\sim_D$. 
\end{definition}

\begin{definition}
    Any vector $\z \in \mathbb{R}^n$ is \emph{homogeneous} if $\exists \zeta \in \mathbb{R}$ such that $z_{(i)} = \zeta$ for each $i \in [n]$.
\end{definition}

\section{Validity}
\label{sec:validity}

In this section we develop broad criteria under which a data-dependent lower bound $\psi$ on a scalar parameter $\pi$ of an unknown distribution $Q \in \mathcal{Q}$ can be considered valid at the $1-\alpha$ level, i.e. when will the bound be no greater than $\pi[Q]$ with probability at least $1-\alpha$. Although fundamentally statistical in nature, our development is purely probabilistic.

\begin{definition}
\label{def:alpha_rej_region}
    For an arbitrary $Q \in \mathcal{Q}$ and $\alpha \in [0,1]$, an $\alpha$-level \emph{rejection region} for $Q$ is any measurable region $R \subseteq \Omega$ such that $Q[R] \leq \alpha$.  
\end{definition}

\begin{definition}
\label{def:alpha_bound}
For any $Q \in \mathcal{Q}$, a \emph{lower confidence bound with miscoverage $\alpha$ on $\pi[Q]$}, or $\alpha$-LCB for brevity, is any measurable function $\psi: \Omega \rightarrow \bar{\mathbb{R}}$ where
    \begin{equation*}
        \{\x \in \Omega: \psi(\x) > \pi[Q]\}
    \end{equation*}
    is an $\alpha$-level rejection region, i.e.,
    \begin{equation*}
        Q[\psi(\X) > \pi[Q]] \leq \alpha.
    \end{equation*}
When $\alpha$ is unspecified, we refer to $\psi$ generically as a statistic.
\end{definition} 

\begin{definition}
\label{eq:obliv_alpha_bound}
    Statistic $\psi : \Omega \rightarrow \bar{\mathbb{R}}$ is a \emph{$\mathcal{Q}$-valid} $\alpha$-LCB if it is an $\alpha$-LCB on $\pi[Q]$ for every $Q \in \mathcal{Q}$. When the set $\mathcal{Q}$ is known implicitly by context we often drop it and write simply that $\psi$ is \emph{a valid $\alpha$-LCB}.
\end{definition}

In the present work, $\alpha$-LCB $\psi$ is never parameterized by an underlying distribution; it is a function only of the sample $\x$.

\begin{definition}
\label{def:rej_map}
    A \emph{model-indexed rejection family} $\mathcal{R}: \mathcal{Q} \rightarrow \mathcal{B}^n$, or \emph{rejection family} for brevity, is a mapping that assigns to each $Q \in \mathcal{Q}$ a rejection region. If, for some $\alpha \in [0,1]$, every rejection region is $\alpha$-level, then $\mathcal{R}$ is said to be an \emph{$\alpha$-level rejection family}.
\end{definition}

\begin{definition}
\label{def:consistent_map_bound}
    For any statistic $\psi: \Omega \rightarrow \bar{\mathbb{R}}$, the rejection family $\mathcal{R}_\psi: \mathcal{Q} \rightarrow \mathcal{B}^n$ \emph{consistent} with $\psi$ is the map giving for each $Q \in \mathcal{Q}$ the rejection region 
    \begin{equation*}
        \mathcal{R}_\psi(Q) \coloneqq \{\x \in \Omega: \psi(\x) > \pi[Q]\}.
    \end{equation*}
\end{definition}

The results in this document present a stand-alone theory of lower confidence bounds. Nevertheless, they can also be interpreted through the lens of hypothesis testing. A $\mathcal{Q}$-valid $\alpha$-LCB may be interpreted as defining, for every finite threshold $t$, a level-$\alpha$ test of the one-sided composite hypothesis
\begin{equation*}
    H_{0,t}: \pi[Q] \leq t
    \qquad \text{against} \qquad
    H_{1,t}: \pi[Q] > t.
\end{equation*}
The test rejects $H_{0,t}$ when $\psi(\x) > t$. As a result, for any $Q \in \mathcal{Q}$ satisfying $H_{0,t}$,
\begin{equation*}
    \{\x \in \Omega : \psi(\x) > t\}
    \subseteq
    \{\x \in \Omega : \psi(\x) > \pi[Q]\},
\end{equation*}
so that
\begin{equation*}
    Q[\psi(\X) > t] \leq \alpha
\end{equation*}
whenever $\psi$ is valid. In particular, for any $Q \in \mathcal{Q}$ with finite $\pi[Q]$, setting $t = \pi[Q]$ makes the rejection region of this test precisely $\mathcal{R}_{\psi}(Q)$. Thus, the model-indexed rejection family records the rejection region obtained at the parameter value associated with each individual model $Q$, while the statistic $\psi$ determines the corresponding nested family of one-sided tests over all thresholds $t$.

\begin{proposition}
\label{prop:alpha_bound_implies_alpha_level}
    If $\psi: \Omega \rightarrow \bar{\mathbb{R}}$ is any valid $\alpha$-LCB, then $\mathcal{R}_\psi$ is an $\alpha$-level rejection family.
\end{proposition}

\begin{proof}
    Since $\psi$ is a valid $\alpha$-LCB, we have that for each $Q \in \mathcal{Q}$,
    \begin{equation*}
        \begin{array}{rcl}
            Q[\mathcal{R}_\psi(Q)] &=& Q[\{\x \in \Omega: \psi(\x) > \pi[Q]\}] \\
            &=& Q[\psi(\X) > \pi[Q]] \\
            &\leq& \alpha
        \end{array}.
    \end{equation*}
\end{proof}

\begin{definition}
\label{def:active}
    Let rejection family $\mathcal{R}: \mathcal{Q} \rightarrow \mathcal{B}^n$ and $\x \in \Omega$ be given. For each $Q \in \mathcal{Q}$, when $\x \not \in \mathcal{R}(Q)$ we say that distribution $Q$ is \emph{active} for $\x$. By $\mathcal{A}(\mathcal{R}, \x)$ we denote the set of distributions active $\x$.  
\end{definition}

\begin{definition}
\label{def:class_bound}
Fix rejection family $\mathcal{R}: \mathcal{Q} \rightarrow \mathcal{B}^n$. The \emph{$\mathcal{Q}$-class bound on $\pi$ with respect to $\mathcal{R}$ given $\x \in \Omega$} is given by
\begin{equation}
\label{eq:candidate_bound}
    M_\pi(\x; \mathcal{Q}, \mathcal{R}) \coloneqq \inf\limits_{Q \in \mathcal{A}(\mathcal{R}, \x)} \pi[Q].
\end{equation}
If $\mathcal{A}(\mathcal{R}, \x)$ is empty, then $M_\pi(\x; \mathcal{Q}, \mathcal{R}) \coloneqq \infty$.
\end{definition}

Going forward, we will assume that the function $\x \mapsto M_\pi(\x; \mathcal{Q}, \mathcal{R})$ is measurable.

\begin{proposition}
\label{prop:consistent_implies_equal}
    For any statistic $\psi: \Omega \rightarrow \bar{\mathbb{R}}$, $\psi(\x) \leq M_\pi(\x; \mathcal{Q}, \mathcal{R}_\psi)$ for each $\x \in \Omega$. 
\end{proposition}
\begin{proof}
    Fix $\x \in \Omega$. We proceed under the assumption that $\mathcal{A}(\mathcal{R}_\psi, \x)$ is non-empty because otherwise the conclusion is clearly true. 
    Let $Q \in \mathcal{A}(\mathcal{R}_\psi, \x)$ be arbitrary. Since $Q$ is active for $\x$, we have by Definition~\ref{def:active} that
    \begin{equation*}
        \x \not \in \mathcal{R}_\psi(Q).
    \end{equation*}
    Thus, by Definition~\ref{def:consistent_map_bound},
    \begin{equation*}
        \psi(\x) \leq \pi[Q].
    \end{equation*}
    Since $Q$ was an arbitrary active distribution, we must also have
    \begin{equation*}
        \begin{array}{rcll}
        \psi(\x) &\leq& \inf\limits_{Q \in \mathcal{A}(\mathcal{R}_\psi, \x)} \pi[Q] & \\
        &=& M_\pi(\x; \mathcal{Q}, \mathcal{R}_\psi) & \text{Definition~\ref{def:class_bound}} \\
        \end{array}.
    \end{equation*}
\end{proof}

\begin{proposition}
\label{prop:bound_implies_rank}
    For any $Q \in \mathcal{Q}$ and rejection family $\mathcal{R}: \mathcal{Q} \rightarrow \mathcal{B}^n$,
    \begin{equation*}
        \{\x \in \Omega: \x \not \in \mathcal{R}(Q) \} \subseteq \{\x \in \Omega: M_\pi(\x; \mathcal{Q}, \mathcal{R}) \leq \pi[Q]\}.
    \end{equation*}
\end{proposition}
\begin{proof}
    By Definition~\ref{def:active} we have that $Q \in \mathcal{A}(\mathcal{R}, \x)$ for each $\x \not \in \mathcal{R}(Q)$. Therefore,
    \begin{equation*}
        M_\pi(\x; \mathcal{Q}, \mathcal{R}) = \inf\limits_{G \in \mathcal{A}(\mathcal{R}, \x)} \pi[G] \leq \pi[Q].
    \end{equation*}
    It follows that $\x \in \{\y \in \Omega: M_\pi(\y; \mathcal{Q}, \mathcal{R}) \leq \pi[Q]\}$, and the conclusion follows.
\end{proof}

\begin{theorem}
\label{eq:general_bound}
Let $\alpha \in [0,1]$ be arbitrary and suppose that $\alpha$-level rejection family $\mathcal{R}: \mathcal{Q} \rightarrow \mathcal{B}^n$ is such that $\x \mapsto M_\pi(\x; \mathcal{Q}, \mathcal{R})$ is measurable. Then $\x \mapsto M_\pi(\x; \mathcal{Q}, \mathcal{R})$ is a valid $\alpha$-LCB on $\pi[Q]$ restricted to $\mathcal{Q}$.
\end{theorem}
\begin{proof}
We must show that, for any $Q \in \mathcal{Q}$, $Q[M_\pi(\X; \mathcal{Q}, \mathcal{R}) > \pi[Q]] \leq \alpha$, or equivalently that $Q[M_\pi(\X; \mathcal{Q}, \mathcal{R}) \leq \pi[Q]] \geq 1-\alpha$. To that end, we have
    \begin{equation}
    \label{eq:alpha_bound}
        \begin{array}{rcll}
            Q[M_\pi(\X; \mathcal{Q}, \mathcal{R}) \leq \pi[Q]] &=& Q[\{\x \in \Omega: M_\pi(\x; \mathcal{Q}, \mathcal{R}) \leq \pi[Q]\}] & \\
            &\geq& Q[\{\x \in \Omega: \x \not \in \mathcal{R}(Q)\}] & \text{by Proposition~\ref{prop:bound_implies_rank}}\\
            &=& Q\left[\X \not \in \mathcal{R}(Q) \right] &  \\
            &=& 1-Q\left[\mathcal{R}(Q) \right] &  \\
            &\geq& 1-\alpha. & \text{by Definition~\ref{def:alpha_rej_region}} \\
        \end{array}
    \end{equation}
\end{proof}

\begin{theorem}
\label{thm:gen_is_opt_alpha}
    For each $\alpha \in [0,1]$ and every valid $\alpha$-LCB, $\psi: \Omega \rightarrow \bar{\mathbb{R}}$, there exists an $\alpha$-level rejection family $\mathcal{R}: \mathcal{Q} \rightarrow \mathcal{B}^n$ such that $\forall \x \in \Omega, \psi(\x) \leq M_\pi(\x; \mathcal{Q}, \mathcal{R})$.
\end{theorem}

\begin{proof}
    Since $\psi$ is a valid $\alpha$-LCB, we know by Proposition~\ref{prop:alpha_bound_implies_alpha_level} that $\mathcal{R}_\psi$ is an $\alpha$-level rejection family. The conclusion then follows directly from Proposition~\ref{prop:consistent_implies_equal}.    
\end{proof}

Theorem~\ref{eq:general_bound} ensures that we can turn any $\alpha$-level rejection family $\mathcal{R}$ into an $\alpha$-LCB $\x \mapsto M_\pi(\x; \mathcal{Q}, \mathcal{R})$ provided that the latter is measurable. 
And Theorem~\ref{thm:gen_is_opt_alpha} proves that every $\alpha$-LCB is itself bounded by $M_\pi(\x; \mathcal{Q}, \mathcal{R})$ for some $\alpha$-level rejection family $\mathcal{R}$.

\section{Conditional Optimality}

Formally, a bound evaluated at every sample in $\Omega$ is a vector, perhaps having uncountable dimension, and as such can be compared to other bounds over the same space. Thus, there exist as many ways to compare bounds as there are ways to compare vectors. In this section we define one in particular: order-conditional optimality; in classical terms, the lower Buehler bound for the given sample order.  

This points to the biggest conceptual difference in moving from the previous section to the present one. In the former, we allowed arbitrary rejection regions for each $Q \in \mathcal{Q}$, while in the latter we will be constrained to a fixed sample order across all distributions in $\mathcal{Q}$. Specifically, in this section a total preorder over samples must first be fixed, any such order is permissible, and samples enter the rejection region for each distribution in this fixed order. We call this a \emph{nested bound} because the set of distributions active for a given sample is always nested with the set active for a sample lower in the total preoder. Somewhat surprisingly, Theorem~\ref{thm:cond_opt} shows that this restriction does not exclude any \emph{good} bounds. In particular, it shows that among all bounds consistent with a given total preorder, the nested bound for that order is dominant. This result was proved previously by Learned-Miller~\cite{learnedmiller2025admissibilityboundsmeandiscrete} in the context of i.i.d. distributions over finite samples spaces.

\begin{definition}
\label{def:consis_preorder}
    We say that statistic $\psi: \Omega \rightarrow \bar{\mathbb{R}}$ is \emph{consistent} with total preorder $D$ if for any $\x, \y \in \Omega$,
    \begin{equation*}
        \x \lesssim_D \y \Rightarrow \psi(\x) \leq \psi(\y).
    \end{equation*}
    If more strongly
    \begin{equation*}
        \x \lesssim_D \y \Leftrightarrow \psi(\x) \leq \psi(\y),
    \end{equation*}
    then we say that $\psi$ \emph{induces} $D$.
\end{definition}

The next two definitions are direct analogs of those given by Learned-Miller~\cite{learnedmiller2025admissibilityboundsmeandiscrete}.

\begin{definition}
\label{def:dominates}
    For statistics $\psi, \psi': \Omega \rightarrow \bar{\mathbb{R}}$ we say that $\psi'$ \emph{dominates} $\psi$, denoted $\psi < \psi'$, if $\forall \x \in \Omega$, $\psi(\x) \leq \psi'(\x)$ and $\exists \x \in \Omega$ such that $\psi(\x) < \psi'(\x)$.
\end{definition}

\begin{definition}
\label{def:cond_opt}
    Fix $\alpha \in [0, 1]$ and a class of statistics $\mathcal{S} \subseteq \{\psi: \Omega \rightarrow \bar{\mathbb{R}}\}$, each a $\mathcal{Q}$-valid $\alpha$-LCB. We say that $\psi \in \mathcal{S}$ is \emph{conditionally optimal} for $\mathcal{S}$ if it dominates every other statistic in $\mathcal{S}$. When we say statistic $\psi$, consistent with total preorder $D$, is conditionally optimal without specifying the class, we mean implicitly that it is conditionally optimal with respect to the set of all $\mathcal{Q}$-valid $\alpha$-LCBs consistent with $D$. 
\end{definition}

Throughout the remainder of this document, we assume that set representation of $D$ is Borel measurable in $\Omega \times \Omega$, i.e.
\begin{equation*}
    \{(\x, \y) \in \Omega \times \Omega: \x \lesssim_D \y \} \in \mathcal{B}^n \otimes \mathcal{B}^n. 
\end{equation*}

\begin{definition}
\label{def:pre_upper}
    The \emph{upper set} associated with sample $\bm x \in \Omega$ and total preorder $D$ is given by 
    \begin{equation*}
        \Omega(\bm x, D) \coloneqq \{\y \in \Omega: \bm x \lesssim_D \bm y\}.
    \end{equation*}
\end{definition}

Going forward, we assume that all the upper sets we work with are measurable.

\begin{definition}
\label{def:nested_order_bound}
    Let $\alpha \in [0,1]$, $\x \in \Omega$, and $D$, a total preorder over samples $\Omega$, be given. The \emph{nested $\mathcal{Q}$-class bound on $\pi$ at level $\alpha$ for $\x$} is given by 
    \begin{equation*}
        \mathcal{N}^\alpha_\pi(\x; \mathcal{Q}, D) \coloneqq \inf_{Q \in \mathcal{Q}: Q[\Omega(\x, D)] > \alpha} \pi[Q],
    \end{equation*}
    If the set $\{Q \in \mathcal{Q}: Q[\Omega(\x, D)] > \alpha\}$ is empty, then $\mathcal{N}^\alpha_\pi(\x; \mathcal{Q}, D) \coloneqq \infty$.
\end{definition}  

\begin{proposition}
\label{prop:nested_consis_with_D}
    Fix $\alpha \in [0,1]$ and total preorder $D$ over samples $\Omega$. If $\x, \y \in \Omega$ are such that $\x \lesssim_D \y$, then
    \begin{equation*}
        \mathcal{N}^\alpha_\pi(\x; \mathcal{Q}, D) \leq \mathcal{N}^\alpha_\pi(\y; \mathcal{Q}, D),
    \end{equation*}
    making the statistic $\x \mapsto \mathcal{N}^\alpha_\pi(\x; \mathcal{Q}, D)$ consistent with $D$.
\end{proposition}

\begin{proof}
    Since $\x \lesssim_D \y$, we know by Definition~\ref{def:pre_upper} that $\Omega(\y, D) \subseteq \Omega(\x, D)$. It follows then that for each $Q \in \mathcal{Q}$, $Q[\Omega(\y, D)] > \alpha \Longrightarrow Q[\Omega(\x, D)] > \alpha$, which in turn implies that
    \begin{equation*}
        \{Q \in \mathcal{Q}: Q[\Omega(\y, D)] > \alpha\} \subseteq \{Q \in \mathcal{Q}: Q[\Omega(\x, D)] > \alpha\}.
    \end{equation*}
    Thus, the conclusion follows by Definition~\ref{def:nested_order_bound} and the fact that the infimum over a set cannot exceed the infimum of a subset of that set.
\end{proof}

\begin{definition}
\label{def:nested_rej_map}
    Given arbitrary $\alpha \in [0,1]$ and total preorder $D$, the \emph{nested rejection family $\mathcal{R}_D: \mathcal{Q} \rightarrow \mathcal{B}^n$ with respect to $D$} defines for each $Q \in \mathcal{Q}$
    \begin{equation}
    \label{eq:lm_reject}
        \mathcal{R}_D^\alpha(Q) \coloneqq \{\x \in \Omega: Q[\Omega(\x, D)] \leq \alpha\}.
    \end{equation}    
\end{definition}

\begin{proposition}
\label{prop:nest_rej_map_alpha}
    For any $\alpha \in [0,1]$ and total preorder $D$ over samples $\Omega$, $\mathcal{R}^\alpha_D$ is an $\alpha$-level rejection family.  
\end{proposition}

\begin{proof}
    Fix arbitrary $Q \in \mathcal{Q}$ and let $R_Q \coloneqq \mathcal{R}^\alpha_D(Q)$. By Definition~\ref{def:alpha_rej_region} it will suffice to show that $Q[R_Q] \leq \alpha$, and we will do so by arguing about the limiting behavior of a finite random sample from $\Omega$ according to the law $Q$. If $Q[R_Q] = 0$, then clearly the result follows for all $\alpha \in [0,1]$. Thus, we proceed under the assumption that $Q[R_Q] > 0$. 
    
    For any integer $m \geq 2$, let $\Z_1, \ldots, \Z_m \overset{\text{ind}}{\sim} Q$ be independent copies of the full observation vector, and let $\mathbb{P}$ be their joint probability law. Define the event
    \begin{equation*}
        A \coloneqq \{\Z_1, \ldots, \Z_m \in R_Q\},
    \end{equation*}
    and, for each $i \in [m]$, the events
    \begin{equation*}
        E_i \coloneqq \{\Z_i \in R_Q \text{~and~} \Z_j \in \Omega(\Z_i, D) \text{~for each~} j \neq i \}.
    \end{equation*}
    Note that by independence of the $Z_i$, 
    \begin{equation}
    \label{eq:prob_A}
        \mathbb{P}[A] = (Q[R_Q])^m,
    \end{equation}
    and
    \begin{equation}
    \label{eq:prob_Ei}
        \mathbb{P}[E_i] \leq \alpha^{m-1} Q[R_Q].
    \end{equation}
    We can see that the latter is true as follows. The joint event defined by $E_i$ decomposes into the probability that $\Z_i \in R_Q$ multiplied by the conditional probability that $\Z_j \in \Omega(\Z_i, D)$, $j \neq i$. Conditional on $\Z_i$ taking on any value $\z \in R_Q$, the other $m-1$ vectors $\Z_j$ remain independent and distributed according to $Q$. Therefore, each belong to $\Omega(\z, D)$ with probability $Q[\Omega(\z, D)]$, which is no greater than $\alpha$ according to Definition~\ref{def:nested_rej_map}. So the conditional probability that they all belong to $\Omega(\z, D)$ is bounded from above by $\alpha^{m-1}$.
    
    Since $D$ is a total preorder, at least one of the $\Z_i$ must be such that $\forall j \in [m], \Z_i \lesssim_D \Z_j$. Suppose that $\Z_j \in R_Q$, for each $j \in [m]$, i.e. event $A$ has occurred. Then there must be some $\Z_i$ among the $m$ such that $\forall j \in [m], \Z_i \in R_Q \text{~and~} \Z_i \lesssim_D \Z_j$. By Definition~\ref{def:pre_upper} this implies that $\Z_j \in \Omega(\Z_i, D)$ for each $\Z_j$, i.e. one of the events $E_i$ has occurred. Therefore, we must have $A \subseteq \bigcup_{i=1}^m E_i$. From the monotonicity of measure we have from Equation~\ref{eq:prob_A} and applying a union bound to Inequality~\ref{eq:prob_Ei} that
    \begin{equation*}
        (Q[R_Q])^m \leq m \alpha^{m-1} Q[R_Q],
    \end{equation*}
    so that
    \begin{equation*}
        Q[R_Q] \leq \alpha m^\frac{1}{m-1}.
    \end{equation*}
    The conclusion follows by allowing $m \rightarrow \infty$. 
\end{proof}

\begin{theorem}
\label{thm:nested_bound_is_alpha}
    For any $\alpha \in [0,1]$ and total preorder $D$ over samples $\Omega$ such that $\x \mapsto \mathcal{N}^\alpha_\pi(\x; \mathcal{Q}, D)$ is measurable, $\x \mapsto \mathcal{N}^\alpha_\pi(\x; \mathcal{Q}, D)$ is a valid $\alpha$-LCB.
\end{theorem}

\begin{proof}
    Applying Definitions~\ref{eq:candidate_bound},~\ref{def:nested_order_bound}, and~\ref{def:nested_rej_map} we have
    \begin{equation}
    \label{eq:N_M_equiv}
        \begin{array}{rcl}
            \mathcal{N}^\alpha_\pi(\x; \mathcal{Q}, D) &=& \inf\{\pi[Q]: Q \in \mathcal{Q}, Q[\Omega(\x, D)] > \alpha\} \\
            &=& \inf\{\pi[Q]: Q \in \mathcal{A}(\mathcal{R}_D^\alpha, \x)\} \\
            &=& M_\pi(\x; \mathcal{Q}, \mathcal{R}_D^\alpha)
    \end{array}.
    \end{equation}
    The conclusion follows by noting that $\mathcal{R}_D^\alpha$ is $\alpha$-level by Proposition~\ref{prop:nest_rej_map_alpha} and applying Theorem~\ref{eq:general_bound}.
\end{proof}

\begin{proposition}
    \label{prop:psi_leq_N}
    Fix $\alpha \in (0, 1)$. If $\psi: \Omega \rightarrow \bar{\mathbb{R}}$ is a $\mathcal{Q}$-valid $\alpha$-LCB consistent with total preorder $D$, then $\forall \x \in \Omega$
    \begin{equation*}
        \psi(\x) \leq \mathcal{N}^\alpha_\pi(\x; \mathcal{Q}, D),
    \end{equation*}
    regardless of the measurability of $\x \mapsto \mathcal{N}^\alpha_\pi(\x; \mathcal{Q}, D)$.
\end{proposition}

\begin{proof}
    Define $\mathcal{Q}_\x \coloneqq \{Q \in \mathcal{Q}: Q[\Omega(\x, D)] > \alpha \}$. If $\mathcal{Q}_\x = \emptyset$, then $\mathcal{N}^\alpha_\pi(\x; \mathcal{Q}, D) = \infty$, and the conclusion follows immediately.

    Now assuming that $\mathcal{Q}_\x \neq \emptyset$, let $Q \in \mathcal{Q}_\x$ be arbitrary. If $\psi(\x) > \pi[Q]$, then because $\psi$ is assumed consistent with $D$ we must have that $\psi(\y) > \pi[Q]$ for all $\y \in \Omega(\x, D)$. But this would imply that
    $Q[\psi(\X) > \pi[Q]] > \alpha$, contradicting the assumption that $\psi$ $\mathcal{Q}$-valid $\alpha$-LCB by Definition~\ref{def:alpha_bound}. Therefore, $\forall Q \in \mathcal{Q}_\x$, $\psi(\x) \leq \pi[Q]$. It follows then that
    \begin{equation}
        \psi(\x) \leq \inf_{Q \in \mathcal{Q}_\x} \pi[Q],
    \end{equation}
    so that $\psi(\x) \leq \mathcal{N}^\alpha_\pi(\x; \mathcal{Q}, D)$ by Definition~\ref{def:nested_order_bound}.
\end{proof}

\begin{theorem}
\label{thm:cond_opt}
    For every $\alpha \in (0, 1)$, if $\x \mapsto \mathcal{N}^\alpha_\pi(\x; \mathcal{Q}, D)$ is measurable, then $\x \mapsto \mathcal{N}^\alpha_\pi(\x; \mathcal{Q}, D)$ is conditionally optimal.
\end{theorem}

\begin{proof}
    By Theorem~\ref{thm:nested_bound_is_alpha}, the measurable $\x \mapsto \mathcal{N}^\alpha_\pi(\x; \mathcal{Q}, D)$ is a $\mathcal{Q}$-valid $\alpha$-LCB. And by Proposition~\ref{prop:nested_consis_with_D}, it is also consistent with $D$. Finally, Proposition~\ref{prop:psi_leq_N} shows that it dominates every other $\mathcal{Q}$-valid $\alpha$-LCB consistent with $D$. Conditional optimality follows then by Definition~\ref{def:cond_opt}.
\end{proof}

Theorem~\ref{thm:cond_opt} proves that no $\alpha$-LCB consistent with total preorder $D$ can improve on the $\alpha$-LCB $\x \mapsto \mathcal{N}^\alpha_\pi(\x; \mathcal{Q}, D)$. Thus, to find the optimal $\alpha$-LCB consistent with order $D$, we need only optimize over rejection regions of the form given by Equation~\ref{eq:lm_reject}.

\section{Two-point Marginals for the Maximum Marginal Mean}
\label{sec:two_point}

Let
\begin{equation*}
    \mathcal{Q}_\text{ind}
    \coloneqq
    \{F_1 \otimes \cdots \otimes F_n:
      F_i \in \mathcal{P}(\mathbb{R}_+),\ i \in [n]\}.   
\end{equation*}
Thus under $\mathcal{Q}_\text{ind}$ the components $X_1, \ldots, X_n$ of random vector $\X$ are independent and non-negative, but not necessarily identically distributed. Throughout this section we assume that $\mathcal{Q}_\text{ind} \subseteq \mathcal{Q}$. In particular, $\mathcal{Q}$ contains every i.i.d. law of the form
\begin{equation*}
    ((1-p)\delta_0 + p\delta_s)^{\otimes n},
    \qquad p \in [0,1],\quad s \in \mathbb R_+.
\end{equation*}
Because $\delta_{x_1}\otimes\cdots\otimes\delta_{x_n}
    \in\mathcal Q_{\mathrm{ind}}\subseteq\mathcal Q$
for every $\x\in\mathbb R_+^n$, the closed model support in this section is $\Omega=\mathbb R_+^n$.

For $Q \in \mathcal{Q}$, take $\X = (X_1, \ldots, X_n) \sim Q$ and define $\mu_i[Q] \coloneqq E_Q[X_i]$ and $\mu_\text{max}[Q] \coloneqq \max_{i \in [n]} \mu_i[Q]$, each taking values in $\bar{\mathbb{R}}_+$. Going forward, we restrict our attention to this maximum marginal mean parameter, i.e. $\pi[Q] \coloneqq \mu_\text{max}[Q]$. The maximum marginal mean captures the common-mean model, in which $\mu_1[Q] = \cdots = \mu_n[Q]$, and the i.i.d. mean as special cases. 

\begin{definition}
\label{def:family_approx}
    Let $\psi: \Omega \rightarrow \bar{\mathbb{R}}_+$ be an arbitrary $\mathcal{Q}$-valid $\alpha$-LCB inducing total preorder $D$. For any sample $\x \in \Omega$, where $\psi(\x) < \infty$, we say that $\psi(\x)$ is \emph{approximated from above}  by a family of laws $Q_\eta \in \mathcal{Q}$ if for each $\eta > 0$, $Q_\eta(\Omega(\x, D)) > \alpha$ and $\psi(\x) \leq \pi[Q_\eta] < \psi(\x) + \eta$.
\end{definition}

Whenever $\psi$ induces total preorder $D$ and $\forall \x \in \Omega$,
\begin{equation*}
    \mathcal{N}^\alpha_\pi(\x;\mathcal{Q},D) = \psi(\x),
\end{equation*}
Definition~\ref{def:family_approx} states that a family of distributions approximates a bound at a given sample if the infimum in Definition~\ref{def:nested_order_bound} can be approached within the smaller class of laws from this family. 

\begin{proposition}
\label{prop:finite_above}
    Fix $\alpha \in [0,1)$ and suppose that $\mathcal{Q}_\text{ind} \subseteq \mathcal{Q}$. If $\psi:\Omega \rightarrow \bar{\mathbb{R}}$ is a $\mathcal{Q}$-valid $\alpha$-LCB on $\mu_\text{max}$, then $\psi(\x)\leq x_{(n)}$ for every $\x\in\Omega$. In particular, if $\psi:\Omega \rightarrow \bar{\mathbb{R}}_+$, then $\psi$ is finite-valued.
\end{proposition}

\begin{proof}
    For each $\x \in \Omega$, the class $\mathcal{Q}_\text{ind}$ contains the distribution $Q' \coloneqq \delta_{x_1} \otimes \cdots \otimes \delta_{x_n}$. Noting that $\mu_\text{max}[Q'] = x_{(n)}$, where $x_{(n)}$ is the largest component of $\x$, Definition~\ref{eq:obliv_alpha_bound} requires that $Q'[\psi(\X) > x_{(n)}] \leq \alpha$. Since $Q' \in \mathcal{Q}_\text{ind} \subseteq \mathcal{Q}$ and $\psi$ is $\mathcal{Q}$-valid, then $\psi(\x)$ cannot exceed $x_{(n)}$. But the choice of $\x$ was arbitrary, so $\forall \x \in \Omega, \psi(\x) \leq x_{(n)}$.
\end{proof}

\begin{theorem}
\label{thm:homo_forces_cop}
    Suppose that $\alpha \in (0, 1)$, $\mathcal{Q}_\text{ind} \subseteq \mathcal{Q}$, and let $\psi:\Omega \rightarrow \bar{\mathbb{R}}_+$ be a  measurable $\mathcal{Q}$-valid $\alpha$-LCB on $\mu_\text{max}$ inducing total preorder $D$ over samples in $\Omega$ and having the property that for any $\s \in \Omega$, homogeneous in $s \geq 0$, $\psi(\s) = s \alpha^{1/n}$. Then we have that for any $\x \in \Omega$,
    \begin{equation*}
        \mathcal{N}^\alpha_{\mu_\text{max}}(\x; \mathcal{Q}, D) = \psi(\x),
    \end{equation*}
    and that $\psi$ is conditionally optimal for the order $D$. Moreover, there exists a family of i.i.d. product laws: $Q_\eta = F_\eta^{\otimes n} \in \mathcal{Q}$, with $\eta > 0$ and $F_\eta$ supported on 0 and one other point from $\mathbb{R}_+$ that approximate $\psi(\x)$ from above. 
\end{theorem}

\begin{proof}
    To begin, we know by Proposition~\ref{prop:finite_above} that for all $\x \in \Omega$, $0 \leq \psi(\x)\leq x_{(n)}<\infty$, making $\psi$ finite valued. Proposition~\ref{prop:psi_leq_N} further establishes that 
    \begin{equation}
    \label{eq:homo_less_than_nested}
        \psi(\x) \leq \mathcal{N}^\alpha_{\mu_\text{max}}(\x; \mathcal{Q}, D). 
    \end{equation}
    Thus, it remains only to prove the opposite inequality. 
    
    Suppose first that $\psi(\x) = \psi(\bm{0}) = 0$, where $\bm{0} \coloneqq (0, \ldots, 0)$. Because $\x \sim_D \bm{0}$ we must have $\bm{0} \in \Omega(\x, D)$, so that the distribution $Q_0 \coloneqq \delta_0^{\otimes n}$ satisfies both $Q_0(\Omega(\x, D)) = 1 > \alpha$ and $\mu_\text{max}[Q_0] = 0$. Thus,
    \begin{equation*}
        \mathcal{N}^\alpha_{\mu_\text{max}}(\x; \mathcal{Q}, D) = 0,
    \end{equation*}
    and both the primary and secondary conclusions of the theorem follow.

    Next, assume that $\psi(\x) > 0$. By assumption, every value in $\mathbb{R}_+$ is achieved by $\psi$ for some homogeneous sample. As a result, there must exist some $\s \in \Omega$, homogeneous in $s = \psi(\x) / \alpha^{1/n}$, such that   
    \begin{equation*}
        \psi(\x) = \psi(\s).
    \end{equation*}
    This means that $\x \sim_{D} \s$ and $\s \in \Omega(\x, D)$. Now for each $\epsilon \in (0, 1-\alpha^{1/n})$ define 
    \begin{equation*}
        F_\epsilon \coloneqq (1- \alpha^{1/n} - \epsilon) \delta_0 + (\alpha^{1/n} + \epsilon) \delta_s,
    \end{equation*}
    and define $Q_\epsilon \coloneqq F_\epsilon^{\otimes n}$. Since $F_\epsilon$ assigns mass $(\alpha^{1/n} + \epsilon)$ to $s$, we must have
    \begin{equation*}
        \alpha < (\alpha^{1/n} + \epsilon)^n \leq Q_\epsilon[\Omega(\x, D)].
    \end{equation*}
    Furthermore, $\mu_\text{max}[Q_\epsilon] = \psi(\x) + s \epsilon$. Therefore,
    \begin{equation}
    \label{eq:nested_less_than_homo}
        \mathcal{N}^\alpha_{\mu_\text{max}}(\x; \mathcal{Q}, D) \leq \psi(\x) + s \epsilon,
    \end{equation}
    and the desired inequality is established by allowing $\epsilon \rightarrow 0$. Since $\psi(\x) = \mathcal{N}^\alpha_{\mu_\text{max}}(\x; \mathcal{Q}, D)$, for all $\x \in \Omega$, and $\psi$ is measurable by assumption, it must also be the case that $\x \mapsto \mathcal{N}^\alpha_{\mu_\text{max}}(\x; \mathcal{Q}, D)$ is measurable. We also have, by Proposition~\ref{prop:nested_consis_with_D}, that $\x \mapsto \mathcal{N}^\alpha_{\mu_\text{max}}(\x; \mathcal{Q}, D)$ is consistent with $D$. Thus, $\psi$ is conditionally optimal for the order $D$ by Theorem~\ref{thm:cond_opt}.
    
    Finally, fix $\eta > 0$. If $\psi(\x) = 0$, then we can take $Q_0 = \delta_0^{\otimes n}$ as our family of approximating distributions. Otherwise, choose $\epsilon \in (0, \min\{1-\alpha^{1/n}, \eta/s\})$. We then have $Q_\epsilon[\Omega(\x, D)] > \alpha$ and 
    \begin{equation*}
        \begin{array}{rcl}
            \psi(\x) &\leq& \mu_\text{max}[Q_\epsilon] \\
            &=& \psi(\x) + s \epsilon \\
            &<& \psi(\x) + \eta \\
        \end{array}.
    \end{equation*}
\end{proof}

\begin{corollary}
\label{cor:dominance_cond}
    Let $\alpha \in (0, 1)$ be fixed and let $\psi: \Omega \rightarrow \bar{\mathbb{R}}_+$ be defined as in Theorem~\ref{thm:homo_forces_cop}. Then $\psi$ cannot be dominated by any $\mathcal{Q}_\text{ind}$-valid $\alpha$-LCB on $\mu_\text{max}$ that cannot be approximated by a family of i.i.d. product laws, each having marginals supported on at most two points.
\end{corollary}

\begin{proof}
    We must prove that any bound $\psi'$ dominating $\psi$ is such that for each $\x \in \Omega$, $\psi'(\x)$ is approximated by an i.i.d. family of distributions supported on two points. By Definition~\ref{def:dominates}, we must have $\psi(\x) \leq \psi'(\x)$ for each $\x \in \Omega$. This is true in particular for every homogeneous sample $\s$. Furthermore, dominance establishes that $\psi'$ must be non-negative and Proposition~\ref{prop:finite_above} further shows that $\psi'$ must be finite. 

    Next, we argue that $\psi'(\s) = s \alpha^{1/n}$ for each $\s \in \Omega$, homogeneous in $s \in \mathbb{R}_+$. Since $\psi'$ is assumed to dominate $\psi$, we must have $s \alpha^{1/n} \leq \psi'(\s)$. Now suppose first that $s = 0$ and let $\bm{0} = (0, \ldots, 0)$. In order for $\psi'$ to be valid we must have $Q[\psi'(\X) > \mu_\text{max}[Q]] \leq \alpha$ for the distribution $Q \coloneq \delta_0^{\otimes n}$. This can occur only if $\psi'(\bm{0}) \leq 0$. By assumption $\psi(\bm{0}) = 0$, and dominance gives $0 \leq \psi'(\bm{0})$. Thus, $\psi'(\bm{0}) = \psi(\bm{0})$.
    
    Next, suppose that $s > 0$ and $s \alpha^{1/n} < \psi'(\s)$. Choose $p \in (\alpha^{1/n}, \min\{1, \psi'(\s)/s\})$ and define $Q_p \coloneqq F_p^{\otimes n}$ where $F_p \coloneqq (1 - p) \delta_0 + p \delta_s$. We have $\mu_\text{max}[Q_p] = p s$, which implies by construction that $\mu_\text{max}[Q_p] < \psi'(\s)$. We also have that $Q_p[\psi'(\X) > \mu_\text{max}[Q_p]] \geq p^n > \alpha$. According to Definition~\ref{eq:obliv_alpha_bound}, this contradicts the assumption that $\psi'$ is $\mathcal{Q}_\text{ind}$-valid. Thus, we must have that $s \alpha^{1/n} \geq \psi'(\s)$, which together with the dominance argument, implies $\psi'(\s) = s \alpha^{1/n}$.
    
    The arguments above establish that the conditions of Theorem~\ref{thm:homo_forces_cop} also apply to $\psi'$ with $\mathcal{Q} = \mathcal{Q}_\text{ind}$ and $D$ being the total preorder induced by $\psi'$, and the conclusion follows.
\end{proof}

\section{Gaffke's Bound}

In this section we prove two facts about Gaffke's bound. First, the bound dominates every other bound consistent with the same total preorder over samples. Second, Gaffke's bound for any particular sample can be approached arbitrarily closely by the maximum marginal means of feasible i.i.d. laws whose common marginal distribution is supported on at most two points. As in Section~\ref{sec:two_point}, we continue to assume that the parameter under consideration is $\pi[Q] = \mu_\text{max}[Q]$, and restrict our attention to $\mathcal{Q}_\text{ind}$, product laws. N.B.: the restriction to product laws is essential for the validity result used below. 

Because $\mathcal{Q}_\text{ind}$ contains the product of arbitrary point-mass distributions, its aggregate support is $\Omega = \mathbb{R}_+^n$. Let $\Delta \coloneqq \{\uu \in [0,1]^n: 0 \leq u_1 \leq \ldots \leq u_n \leq 1\}$ be the simplex of $n$ uniform order statistics, and let $\lambda$ be the normalized Lebesgue measure on $\Delta$. For $\x, \y \in \Omega$, we write $\x \leq \y$ when $x_{(i)} \leq y_{(i)}$ for each $i \in [n]$, and we write $\x < \y$ when at least one of those inequalities is strict, where $x_{(i)}$ represents the $i$th entry of $\x$ when its elements are arranged in increasing order.

\begin{definition}
\label{def:consv_compl}
    For arbitrary $\x \in \Omega$ and $\uu \in \Delta$, the \emph{conservative completion} of $\x$ with respect to $\uu$ is given by
    \begin{equation*}
        c(\x, \uu) \equiv \sum\limits_{i=1}^n x_{(i)} (u_{i+1}-u_{i}),
    \end{equation*}
    where we define $u_{n+1} \equiv 1$.
\end{definition}

\begin{definition}
\label{eq:gaffke_rej}
    For arbitrary $\x \in \Omega$ and $\uu \in \Delta$, the \emph{Gaffke sublevel region} is given by
    \begin{equation*}
        R^*(\x, \uu) \equiv \{\w \in \Delta: c(\x, \w) \leq c(\x, \uu)\}
    \end{equation*}
\end{definition}

\begin{definition}[Gaffke's Bound~\cite{gaffke2005three}]
\label{def:gaffke}
    For each $\alpha \in [0,1)$ and $\x \in \Omega$, 
    \begin{equation*}
        \mathcal{G}^\alpha(\x) \equiv \inf \left\{c(\x, \uu): \uu \in \Delta, \lambda[R^*(\x, \uu)] > \alpha \right\}.
    \end{equation*} 
    Furthermore, we define $\mathcal{G}^1(\x) \coloneqq \infty$.
\end{definition}

\begin{definition}
\label{def:gaffkes_order}
    For fixed $\alpha \in [0,1]$, the \emph{Gaffke order} $D^*_\alpha$ is the total preorder defined such that for all $\x, \y \in \Omega$, 
    \begin{equation*}
        \x \lesssim_{D^*_\alpha} \y \Leftrightarrow \mathcal{G}^\alpha(\x) \leq \mathcal{G}^\alpha(\y).
    \end{equation*}
\end{definition}

\begin{remark}
    The measurability of $\mathcal{G}^\alpha$ and the graph of $D^*_\alpha$ can be established using standard arguments, which we omit here.  
\end{remark}

\begin{theorem}[Vlassis-Thomas, confidence-bound form]
\label{thm:vla_tho}
Take $Q = F_1 \otimes \cdots \otimes F_n$ to be in $\mathcal{Q}_\text{ind}$ and let $\X \sim Q$. 
Then for any $\alpha \in [0, 1]$, 
\begin{equation*}
    Q[\mathcal{G}^\alpha(\X) > \mu_\text{max}[Q]] \leq \alpha,
\end{equation*}
i.e. Gaffke's bound $\mathcal{G}^\alpha$ is a valid $\alpha$-LCB.
\end{theorem}

\begin{proof}
    Fix $Q \in \mathcal{Q}_\text{ind}$. The conclusion clearly holds if either $\alpha = 1$ or $\mu_\text{max}[Q] = 0$, so we proceed under the assumption that $\alpha < 1$ and $\mu_\text{max}[Q] > 0$. 
    
    We first handle the case where $\mu_\text{max}[Q] = \infty$. For each $\x \in \Omega$, $c(\x, \uu)$ achieves its maximum value of $x_{(n)}$ when $\uu = \bm{0}$. Thus, $\mathcal{G}^\alpha(\x) < \infty$. Therefore, $Q[\mathcal{G}^\alpha(\X) > \mu_\text{max}[Q]] = 0$, making $\mathcal{G}^\alpha$ an $\alpha$-LCB for any $\alpha \in [0, 1]$.
    
    Next, we turn to the case where $\mu_\text{max}[Q]$ is finite. For arbitrary $t \in \mathbb{R}_+$ define
    \begin{equation*}
        K_t(\x) \coloneqq \lambda [\uu \in \Delta: c(\x, \uu) \leq t].
    \end{equation*}
    Taking $\D = (D_0, D_1, \ldots, D_n) \sim \texttt{Dir}(1, \ldots, 1)$, and $\mathbb{P}_\D$ its probability measure, we have the following equivalence by the exchangeability of the entries of $D$ as well as its interpretation as uniform spacings. 
    \begin{equation*}
        K_t(\x) = \mathbb{P}_\D \left[ \sum_{i=1}^n x_i D_i  \leq t \right].
    \end{equation*}
    Next, define $\Y$ so that $Y_i = X_i / \mu_\text{max}[Q]$. The $Y_i$ are independent, non-negative, and the mean of each is no greater than 1. Moreover, $K_{\mu_\text{max}[Q]}(\X) = K_1(\Y)$ by construction. It follows from the Vlassis-Thomas Theorem~\cite{vlassis2026exact} that
    \begin{equation}
    \label{eq:vt_extended}
        Q[K_{\mu_\text{max}[Q]}(\X) \leq \alpha] \leq \alpha.
    \end{equation}
    Furthermore, by Definition~\ref{eq:gaffke_rej}  we have that $\forall \uu \in \Delta$,
    \begin{equation}
    \label{eq:R_to_K}
        \lambda[R^*(\x, \uu)] = K_{c(\x, \uu)}(\x).
    \end{equation}
    Notice that by choosing $\uu = (r, \ldots, r)$, $r \in [0, 1]$, $c(\x, \uu)$ can achieve every value for $t$ in the range $[0, x_{(n)}]$, which surely contains the infimum of values $t$ such that $K_t(\x) > \alpha$. Thus, by applying Equation~\ref{eq:R_to_K} to Definition~\ref{def:gaffke} we have
    \begin{equation}
    \label{eq:VT_to_Gaffke}
        \mathcal{G}^\alpha(\x) = \inf \{t \geq 0: K_t(\x) > \alpha \}.
    \end{equation}
    
    Now suppose that $\mathcal{G}^\alpha(\x) > t$. Then according to Equation~\ref{eq:VT_to_Gaffke} it must also be the case that $K_t(\x) \leq \alpha$, i.e.,
    \begin{equation}
    \label{eq:G_contains_K}
        \{\x: \mathcal{G}^\alpha(\x) > t\} \subseteq \{\x: K_t(\x) \leq \alpha\}.
    \end{equation}
    It follows then that
    \begin{equation*}
        \begin{array}{rcll}
        Q[\mathcal{G}^\alpha(\X) > \mu_\text{max}[Q]] &\leq& Q[K_{\mu_\text{max}[Q]}(\X) \leq \alpha] & \text{Equation~\ref{eq:G_contains_K}} \\
        &\leq& \alpha & \text{Inequality~\ref{eq:vt_extended}}
        \end{array}.
    \end{equation*}
    Thus, the conclusion follows according to Definition~\ref{def:alpha_bound}.
\end{proof}

\begin{remark}
    For heterogeneous marginal means, Theorem~\ref{thm:vla_tho} controls $\mu_\text{max}$, not the average marginal mean $\frac{1}{n}\sum_{i \in [n]} \mu_i[Q]$. An average-mean interpretation therefore requires either the common-mean assumption or an additional restriction relating the marginal means.  
\end{remark}

\begin{proposition}
\label{prop:c_monotone_in_x}
    For fixed $\uu \in \Delta$ and any $\x, \y \in \Omega$, if $\x \leq \y$, then $c(\x, \uu) \leq c(\y, \uu)$. 
\end{proposition}

\begin{proposition}
\label{prop:gaffke_equiv}
    For each $\alpha \in [0,1]$, $\x \in \Omega$, and $R \subseteq \Delta$ a closed set with $\lambda[R] > \alpha$,
    \begin{equation*}
        \mathcal{G}^\alpha(\x) \leq  \max_{\uu \in R} ~c(\x, \uu).
    \end{equation*}
\end{proposition}

\begin{proof}
    Suppose for the sake of contradiction that there exists closed set $R \subseteq \Delta$ with $\lambda[R] > \alpha$ such that for all $\uu \in R$, $\mathcal{G}^\alpha(\x) > c(\x, \uu)$. This property holds in particular for the $\uu' \in R$ that achieves the maximum value for $c(\x, \uu)$. But by Definition~\ref{eq:gaffke_rej} we also have $R \subseteq R^*(\x, \uu')$ so that
    \begin{equation*}
        \alpha < \lambda[R] \leq \lambda[R^*(\x, \uu')].
    \end{equation*}
    The conclusion follows by Definition~\ref{def:gaffke}.
\end{proof}

\begin{proposition}
\label{prop:monotonocity_of_gaffke}
    For any $\alpha \in (0,1)$ and $\x, \y \in \Omega$,
    \begin{equation*}
        \x \leq \y \Rightarrow \mathcal{G}^\alpha(\x) \leq \mathcal{G}^\alpha(\y).
    \end{equation*}
\end{proposition}

\begin{proof}    
    Definition~\ref{def:gaffke} ensures that for each $\epsilon > 0$, there exists $\uu_\epsilon \in \Delta$, 
    such that 
    \begin{equation*}
         c(\y, \uu_\epsilon) \leq \mathcal{G}^\alpha(\y) + \epsilon.
    \end{equation*} 
    with $\lambda[R^*(\y, \uu_\epsilon)] > \alpha$. 
    By Proposition~\ref{prop:gaffke_equiv}, it must be the case that for $\w' \in R^*(\y, \uu_\epsilon)$ maximizing $c(\x, \uu)$ w.r.t $\uu$,
    \begin{equation*}
        \mathcal{G}^\alpha(\x) \leq c(\x, \w').
    \end{equation*}
    But by Proposition~\ref{prop:c_monotone_in_x}, and because $\x \leq \y$, it must also be the case that for all $\w \in R^*(\y, \uu_\epsilon)$
    \begin{equation*}
        c(\x, \w) \leq c(\y, \w).
    \end{equation*}
    This is true in particular for $\w'$. It follows then that
    \begin{equation*}
        \mathcal{G}^\alpha(\x) \leq c(\x, \w') \leq c(\y, \w') \leq \max \{c(\y, \w): \w \in R^*(\y, \uu_\epsilon)\} = c(\y, \uu_\epsilon) \leq \mathcal{G}^\alpha(\y) + \epsilon.
    \end{equation*}
    The conclusion follows by allowing $\epsilon \rightarrow 0$.
\end{proof}

\begin{proposition}[Learned-Miller and Thomas~\cite{learnedmiller2019newconfidence}]
\label{prop:gaffke_homogeneous}
    For fixed $\alpha \in (0, 1)$ and sample $\x$, homogeneous in $x$, $\mathcal{G}^\alpha(\x) = x \alpha^{1/n}$. 
\end{proposition}

\begin{proof}
    If $x = 0$, then it's clear that $c(\x, \uu) = 0$ for every $\uu \in \Delta$, making $\mathcal{G}^\alpha(\x) = 0$ as well, so that the conclusion follows immediately.
    Now assume that $x > 0$. Since $\x$ is  homogeneous, for each $\vv \in \Delta$ we have
    \begin{equation}
    \label{eq:cc_homogeneous}
        c(\x, \vv) = x(1-v_1).
    \end{equation}
    Thus, 
    \begin{equation*}
        R^*(\x, \uu) = \{\w \in \Delta: w_1 \geq u_1\}.
    \end{equation*}
    It is well known that the first uniform order statistic $U_1$ is distributed as $\texttt{Beta}(1, n)$ with CDF equal to $t \mapsto 1 - (1 - t)^n$. Therefore, 
    \begin{equation*}
        \lambda[R^*(\x, \uu)] = (1 - u_1)^n.
    \end{equation*}
    For any $\uu \in \Delta$ to achieve $\lambda[R^*(\x, \uu)] > \alpha$, it must be the case then that
    \begin{equation}
        \label{eq:homo_u_1}
        1 - u_1 > \alpha^{1/n}.
    \end{equation}
    Applying this to Equation~\ref{eq:cc_homogeneous}, we have
    \begin{equation}
    \label{eq:gaffke_homo}
        \begin{array}{rcll}
        \mathcal{G}^\alpha(\x) &=& \inf\{c(\x, \uu): \lambda[R^*(\x, \uu)] > \alpha\} & \\
        &=& \inf\{c(\x, \uu): 1 - u_1 > \alpha^{1/n}\} & \text{Equation~\ref{eq:homo_u_1}}\\ 
        &=& \inf\{x(1-u_1): u_1 < 1-\alpha^{1/n}\} & \text{Equation~\ref{eq:cc_homogeneous}} \\ 
        &=& x \alpha^{1/n}
        \end{array}.
    \end{equation}
\end{proof}

\begin{theorem}
\label{thm:gaffke_eq_cop}
    If $\alpha \in (0, 1)$ and $\x \in \Omega$ is any sample, then
    \begin{equation*}
        \mathcal{N}^\alpha_{\mu_\text{max}}(\x; \mathcal{Q}_\text{ind}, D^*_\alpha) = \mathcal{G}^\alpha(\x), 
    \end{equation*}  
    and there exists a family of i.i.d. product laws: $Q_\eta = F_\eta^{\otimes n} \in \mathcal{Q}_\text{ind}$, with $\eta > 0$ and $F_\eta$ supported on 0 and one other point from $\mathbb{R}_+$ that approximate $\mathcal{G}^\alpha(\x)$ from above. Consequently, $\mathcal{G}^\alpha$ is conditionally optimal for the order $D_\alpha^*$ over $\mathcal{Q}_\text{ind}$.
\end{theorem}

\begin{proof}
    It is clear from Definitions~\ref{def:consv_compl} and~\ref{def:gaffke} that $\mathcal{G}^\alpha$ is non-negative.
    We also have by Theorem~\ref{thm:vla_tho} that $\mathcal{G}^\alpha$ is a $\mathcal{Q}_\text{ind}$-valid $\alpha$-LCB on $\mu_\text{max}$, and it induces order $D^*_\alpha$ by Definition~\ref{def:gaffkes_order}. Furthermore, Equation~\ref{eq:gaffke_homo} establishes that for each $s \in \mathbb{R}_+$,  $\mathcal{G}^\alpha(\s) = s \alpha^{1/n}$ for the sample $\s$ homogeneous in $s$.  Therefore, the conclusion follows from Theorem~\ref{thm:homo_forces_cop}. 
\end{proof}

Theorem~\ref{thm:gaffke_eq_cop} establishes that Gaffke's bound is conditionally optimal for the $\mu_\text{max}$ parameter relative to the broad class of all product laws, $\mathcal{Q}_\text{ind}$. The next corollary shows that it is also conditionally optimal for the specialized classes: independent with common marginal mean and i.i.d., where in both cases $\mu_\text{max}$ reduces to the common mean of the marginal distributions.  

\begin{corollary}
\label{cor:submodel_optimality}
Fix $\alpha \in (0, 1)$ and let $\mathcal Q' \subseteq \mathcal{Q}_\text{ind}$ have closed model
support $\Omega = \mathbb{R}_+^n$ and contain every i.i.d. law with two-point marginals
\begin{equation*}
    ((1-p)\delta_0 + p\delta_s)^{\otimes n},
    \qquad p\in[0,1], \quad s \geq 0.    
\end{equation*}
Then, for every $\x \in \Omega$,
\begin{equation*}
    \mathcal{N}^\alpha_{\mu_\text{max}}(\x; \mathcal{Q}', D_\alpha^*) = \mathcal{G}^\alpha(\x).
\end{equation*}
As a result, $\mathcal{G}^\alpha$ is conditionally optimal for
$D_\alpha^*$ over $\mathcal{Q}'$.
\end{corollary}

\begin{proof}
Since $\mathcal{Q}' \subseteq \mathcal{Q}_\text{ind}$,
\begin{equation*}
    \mathcal{N}^\alpha_{\mu_\text{max}}(\x; \mathcal{Q}',D_\alpha^*) \geq \mathcal{N}^\alpha_{\mu_\text{max}}(\x; \mathcal{Q}_\text{ind},D_\alpha^*) = \mathcal{G}^\alpha(\x).
\end{equation*}
On the other hand, for every $\eta > 0$, Theorem~\ref{thm:gaffke_eq_cop} shows that there exists an i.i.d.  law $Q_\eta \in \mathcal{Q}'$ with two-point marginals satisfying $Q_\eta[\Omega(\x,D_\alpha^*)] > \alpha$ and $\mu_{\max}[Q_\eta] < \mathcal{G}^\alpha(\x) + \eta$. Therefore
\begin{equation*}
    \mathcal N^\alpha_{\mu_{\max}}
      (\x;\mathcal Q',D_\alpha^*)
    \leq
    \mathcal G^\alpha(\x)+\eta.
\end{equation*}
Equality follows by allowing $\eta \rightarrow 0$ and conditional optimality follows by Theorem~\ref{thm:cond_opt}.
\end{proof}

\section{Conclusion}

We have presented a framework for reasoning about lower confidence bounds (LCBs) on scalar parameters of probability laws from broad classes of distributions including all joint,  product, and i.i.d. laws. In the setting of product laws and the maximum marginal mean parameter, we showed that each such LCB is conditionally optimal with respect to the total preorder it induces and are approximated from above by a family of product laws whose marginals are supported on at most two points. We went on to show that Gaffke's bound fits these same criteria. 

\section{Disclosure}

AI tools were used in the preparation of this document. The authors assume sole responsibility for any errors or omissions.

\printbibliography

\end{document}